\documentclass{amsart}

\usepackage[T1]{fontenc}
\usepackage{lmodern}
\usepackage{comment}

\usepackage[francais,english]{babel}

\usepackage{latexsym}
\usepackage{amssymb}

\usepackage[centering,textheight=50.5pc,textwidth=30pc]{geometry}
\usepackage[pdfstartview=XYZ]{hyperref}
\theoremstyle{plain}
\newtheorem{theorem}{Theorem}[section]

\newtheorem{lemma}[theorem]{Lemma}

\theoremstyle{definition}

\theoremstyle{remark}
\newtheorem{remark}[theorem]{Remark}

\theoremstyle{conjecture}

\def\R{\mathbb{R}}

\def\Q{\mathbb{Q}}

\def\F{\mathbb{F}}
\def\O{\mathbb{O}}
\newcommand{\rank} {\operatorname{rank}}
\newcommand{\circm} {\operatorname{circm}}

\title[ Ryser's Conjecture]
{On Ryser's Conjecture}

\date{\today}

\author[A. Domic]{Antun Domic}

\address{Kepler Computing and Stanford University\\
\\
\\
}

\email{domic@stanford.edu}

\author[L. H. Gallardo]{Luis H. Gallardo}

\address{Univ. Brest,
UMR CNRS 6205 \\
Laboratoire de Math\'ematiques de Bretagne Atlantique\\
6, Av. Le Gorgeu, C.S. 93837, Cedex 3, F-29238 Brest \\
France}

\email{Luis.Gallardo@univ-brest.fr}



\begin{document}



\begin{abstract}
We provide new sufficient conditions under which Ryser's conjecture holds.
\end{abstract}

\maketitle

\section{Introduction}
\label{intro}
Let A be a square matrix of order $n$ with complex coefficients. We denote by $A^{*}$ the conjugate transpose of the matrix $A$. For any pair of rows $R_i, R_j$ of a matrix $A$ with real coefficients, we denote by $\langle R_{i}, R_{j} \rangle$ the usual scalar product.  
A \emph{circulant} matrix $A := \circm(a_1,\ldots, a_n)$ of order $n$ is a matrix of order $n$ whose first row is $[a_1, \ldots, a_{n}]$ and in which each row after the first is obtained by a cyclic shift of its predecessor by one position. For example, the second row $R_2$ of $A$ is $[a_{n}, a_1, \ldots, a_{n-1}]$. With $\pi := \circm(0,1,0,\ldots,0)$, one has that $R_2$ is the first row of $\pi A$.  

A \emph{Hadamard} matrix $H$ of order $n$ is a matrix of order $n$ with entries in $\{-1,1\}$ such that $K := \frac{H}{\sqrt{n}}$ is orthogonal. A \emph{circulant Hadamard} matrix of order $n$ is a circulant matrix that is also Hadamard. Besides the two trivial matrices of order $1$,  
$H_1 :=\circm(1)$ and $H_2 := -H_1$, the remaining $8$ known circulant Hadamard matrices are $ H_3 :=\circm(1,-1,-1,-1), H_4 := -H_3, H_5 :=\circm(-1,1,-1,-1), H_6 := -H_5, $ $H_7 := \circm(-1,-1,1,-1), H_8 := -H_7,
 H_9 := \circm(-1,-1,-1,1), H_{10} := -H_9.$

If $H = \circm(h_1,\ldots,h_n)$ is a circulant Hadamard matrix of order $n$, then its \emph{representer} polynomial is
\[
R(x) := h_1 + h_2 x + \cdots + h_n x^{n-1}.
\]

No other circulant Hadamard matrices are known. Ryser conjectured in 1963 (see \cite{Ryser}, \cite[p.~97]{Davis}) that no such matrices exist when $n > 4$.  
Previous work on the conjecture include \cite{EGR,EGRcomb,JedwabLloyd,luisAMEN,luisEJC,luisAMEN1,luisMC2,Mato,NG,Leung,Turyn}.

Let $H = \circm(h_1,\ldots,h_n)$ be a circulant Hadamard matrix of order $n \geq 4$.  
Define
\[
E_1 := \circm(h_1,h_3,h_5,\ldots,h_{n-1}), \quad
E_2 := \circm(h_2,h_4,h_6,\ldots,h_{n}).
\]
When $n=4$, both $E_1$ and $E_2$ have rank equal to $1$ (as their determinant is 0, see the matrices $H_3, \ldots, H_{10}$).

Two other interesting properties of $E_1$ and $E_2$ when $n=4$ are the following:
\begin{remark}
\label{projMcW}
Let $K_1 := \frac{E_1+E_2}{2}$ and $K_2 := \frac{E_1-E_2}{2}$. Then:
\begin{itemize}
\item[\rm{(a)}]
$K_1$ and $K_2$ are projections.
\item[\rm{(b)}]
$K_1$ and $K_2$, reduced modulo $2$, are symmetric orthogonal matrices.
\end{itemize}
\end{remark}

Property~(b) in Remark~\ref{projMcW} is significant because of a result of McWilliams \cite[Corollary~1.8]{jessieMW} (see also Lemma~\ref{jessie}) which states that for $n>2$, the orthogonal group $\O(n,\F_2)$ consists only of the identity matrix. In other words, $n=4$ is the only case in which (b) holds.

When a circulant Hadamard matrix of order $n \geq 4$ is such that both $E_1$ and $E_2$ are non-singular, it follows from \cite{luisAMEN1} and our Lemma~\ref{graphR} (see Section~\ref{toolsBM}) that $n$ must equal $4$. The rank-$1$ property that we consider in the present paper is the first natural case to study in order to make progress on the open case when at least one of $E_1, E_2$ is singular.

In this paper we prove two results. Essentially, the first shows that for circulant Hadamard matrices of order $n$, the rank-$1$ property and properties~(a) and~(b) in Remark~\ref{projMcW} hold for $n=4$ and do not hold when $n>4$. Our second main result, which follows from a simple characterization of circulant matrices of rank $\leq 2$ together with the Plotkin bound for codewords, proves the more general statement that there is no circulant Hadamard matrix of order $>4$ in which either one of the submatrices $E_1, E_2$ has rank $1$ while the other has arbitrary rank, or both have rank $2$.

More precisely, in this paper we prove the following two theorems.

\begin{theorem}
\label{HadBM}
Let $H =  \circm(h_1,\ldots,h_n)$ be a circulant Hadamard matrix of order $n \geq 4$. 
Let $E_1 := \circm(h_1,h_3,h_5,\ldots,h_{n-1})$  and  $E_2 := \circm(h_2,h_4,h_6,\ldots,h_{n})$. We denote by $R_j$, respectively $S_j$, the rows of $E_1$ and $E_2$, for $j=1,\ldots,n/2$.  
Let $G_1 := E_1 E_1^{*}$ and $G_2 := E_2 E_2^{*}$ be the Gram matrices of $E_1$ and $E_2$.  
Assume that one of the following statements holds:
\begin{itemize}
\item[\rm{(a)}]
$\rank(E_1) = 1 = \rank(E_2)$.
\item[\rm{(b)}]
All entries of $G_1$ and $G_2$ have the same absolute value.
\item[\rm{(c)}]
$R_1$ and $R_2$ are linearly dependent over $\Q$, and $S_1$ and $S_2$ are linearly dependent over $\Q$. 
\item[\rm{(d)}]
$\vert \langle R_{1}, R_{2} \rangle \vert = \vert \langle R_{1}, R_{1} \rangle \vert$, and $\vert \langle S_{1}, S_{2} \rangle \vert = \vert \langle S_{1}, S_{1} \rangle \vert$.
\end{itemize}
Then $n=4$. 
\end{theorem}

\begin{remark}
\label{neq4}
All statements (a), (b), (c), and (d) of the theorem hold when $n=4$. Moreover, for any entry $e$ of $G_1$ or $G_2$, we have $\vert e \vert = 2 = n/2$.
\end{remark}

\begin{remark}
\label{goyeneche}
Our results also cover the recent generalization of the circulant Hadamard conjecture \cite{GoyenecheT} by Goyeneche and Turek, since a generalized circulant Hadamard matrix $H(d) := \circm(d,h_2,\ldots,h_n)$ satisfying our hypotheses necessarily has $\vert d\vert = 1$.
\end{remark}

Our second main result is as follows:
\begin{theorem}
\label{PlotkinR}
With the notation of Theorem~\ref{HadBM}, there is no circulant Hadamard matrix $H$ of order $n > 4$ in the following two cases:
\begin{itemize}
\item[\rm{(a)}]
One of $E_1, E_2$ has rank $1$.
\item[\rm{(b)}]
Both $E_1$ and $E_2$ have rank $2$.
\end{itemize}
\end{theorem}

The needed preliminaries for the proofs are given in Section~\ref{toolsBM}. The proof of Theorem~\ref{HadBM} is presented in Section~\ref{HaddoneBM}.  The proof of Theorem~\ref{PlotkinR} is presented in Section~\ref{Hdone2}.

\section{Preliminaries}
\label{toolsBM}

Our first lemma is well known and easy to prove.

\begin{lemma}
\label{crankreal}
Let $A$ be a real matrix. Then
$\rank(A^{*}A) = \rank(A A^{*}) = \rank(A) = \rank(A^{*})$.
\end{lemma}

The following lemma (\cite{Davis}) is useful.

\begin{lemma}
\label{circulant}
Let $C := \circm(c_{1},\ldots,c_{k})$ be a circulant matrix of order $k$. Let $v := [1, \ldots, 1] \in \R^k$. Then $v$ is an eigenvector of $C$ 
with associated eigenvalue $\lambda := c_{1} + \cdots + c_{k}$.
\end{lemma}

The following is a result of MacWilliams \cite{jessieMW}. It was already used in
\cite[Theorem~1]{luisAMEN2}.

\begin{lemma}
\label{jessie}
The only circulant, symmetric, and orthogonal matrix over the binary field $\F_2$ of order $n>2$ is the identity matrix $I_n$.
\end{lemma}

The next lemma (see \cite[Lemma~8.6]{Serre}) is frequently used in the theory of group representations. Here it is useful for the proof of part~(d) of Theorem~\ref{HadBM}.

\begin{lemma}
\label{cauchyschw}
Let $c_1, \ldots, c_{\ell}$ be $\ell$ complex numbers of absolute value $1$.
If
\[
\vert c_{1}+\cdots + c_{\ell} \vert = \ell,
\]
then $c_1 = \cdots = c_{\ell}$.
\end{lemma}

The following lemma gives some properties of certain circulant matrices of rank~$1$.

\begin{lemma}
\label{circrank1}
Let $C := \circm(c_{1},\ldots,c_{2k})$ be a circulant matrix of order $2k$ and rank~$1$ such that $c_j^2=1$ for all $j=1,\ldots,2k$.
Then for all $s=1,\ldots,k$ we have
$c_{2s-1}=a$ and $c_{2s}=b$, with $\{a,b\} \subseteq \{-1,1\}$. Moreover, if $a \neq b$ then
\[
c_1+c_2+\cdots+c_{2k}=0.
\]
\end{lemma}

\begin{proof}
Since $\rank(C)=1$, the second row $R_2$ of $C$ is a multiple of the first row $R_1$. As each entry of $C$ is $1$ or $-1$, there are two cases to consider:  
(a) $R_1=R_2$ and (b) $R_1 = -R_2$.  

In case~(a), comparing the first entries of $R_1$ and $R_2$ gives $c_{2k}=c_{1}$, and comparing the others shows $c_1=c_{2}=c_{3}= \cdots = c_{2k-1}$. Thus all entries of $R_1$ are equal, so $a=b \in \{-1,1\}$.  

In case~(b), a similar argument gives the stated result.  

If $a \neq b$, then $\{a,b\} =\{-1,1\}$, so $c_1+c_2+\cdots+c_{2k}=a+b =0$.
\end{proof}

The following lemma gives some properties of certain circulant matrices of rank~$2$.

\begin{lemma}
\label{circrank2}
Let $D := \circm(d_{1},\ldots,d_{2k})$ be a circulant matrix of order $2k \geq 6$ and rank~$2$ such that $d_j^2=1$ for all $j=1,\ldots,2k$. Assume that the first two rows of $D$,
\[
v = (d_1,d_2,d_3,\ldots,d_{2k-2},d_{2k-1},d_{2k}),
\]
and
\[
w =(d_{2k},d_1,d_2,\ldots,d_{2k-3},d_{2k-2},d_{2k-1}),
\]
are $\Q$-linearly independent. Let $u = a v + b w$ be any row of $D$, where $a,b \in \Q$. Assume that
\[
s := d_1+d_2+\cdots + d_{2k-1}+d_{2k} \neq 0.
\]
Then $a+b = 1$.
\end{lemma}

\begin{proof}
Without loss of generality, we may assume that $u$ is the third row of $D$, so that
\begin{equation}
\label{row3}
u = (d_{2k-1},d_{2k},d_1,\ldots,d_{2k-4},d_{2k-3},d_{2k-2}).
\end{equation}
Since $u=av+bw$, with indices taken modulo $2k$, we have for all $j=1,\ldots,2k$:
\begin{equation}
\label{djis}
d_j = a\, d_{j+2} + b\, d_{j+1}.
\end{equation}
Summing over all $j$ in \eqref{djis} gives
\begin{equation}
\label{eses}
s = a\, s + b\, s.
\end{equation}
As $s \neq 0$, equation~\eqref{eses} implies that $a+b=1$, as claimed.
\end{proof}

\begin{lemma}
\label{klmrank2}
Let $h_k,h_{\ell},h_m \in \{-1,1\}$. Assume that for some rational numbers $a,b \in \Q$ satisfying
\begin{equation}
\label{ab1}
a+b = 1,
\end{equation}
we have
\begin{equation}
\label{abklm}
h_k = a\, h_{\ell} + b\, h_m.
\end{equation}
Then either
\begin{equation}
\label{abklm1}
h_k =  h_{\ell} = h_m,
\end{equation}
or
\begin{equation}
\label{abklm2}
(a,b) \in \{(0,1),(1,0)\}.
\end{equation}
\end{lemma}

\begin{proof}
Assume first that $ab \neq 0$.  
It suffices to prove that $h_{\ell} = h_m$, since if this holds then \eqref{abklm} and \eqref{ab1} give
\begin{equation}
\label{elem}
h_k =  (a+b) h_{\ell} = h_{\ell}.
\end{equation}
To prove $h_{\ell} = h_m$, assume without loss of generality that
\begin{equation}
\label{kun}
 h_{\ell} = 1, \quad h_m = -1.
\end{equation}
From \eqref{abklm} we obtain
\begin{equation}
\label{aminusb}
h_k = a - b.
\end{equation}
If $h_k=1$, then together with \eqref{ab1} we find $b=0$, a contradiction to $ab \neq 0$.  
Similarly, if $h_k = -1$ we find $a=0$, again a contradiction.  
Thus $h_{\ell}=h_m$.  

If $ab=0$, then \eqref{ab1} implies \eqref{abklm2}.
\end{proof}

\begin{lemma}
\label{uu}
Let $C = \circm(c_1,\ldots,c_m)$ be a circulant matrix such that two consecutive rows are equal. Then
\begin{equation}
\label{alleq}
c_1 = \ldots = c_m.
\end{equation}
\end{lemma}

\begin{proof}
Without loss of generality, assume that the first two rows of $C$ are equal. Then we have $c_j = c_{j-1}$ for all $j$, with indices taken modulo $m$. This proves the claim.
\end{proof}

\begin{lemma}
\label{uvuu}
Let $C = \circm(c_1,\ldots,c_{2k})$ be a circulant matrix with entries in $\{-1,1\}$, of order $2k \geq 6$, such that the first row of $C$ is equal to the third. Let $s = c_1 + \cdots + c_{2k}$. Then
\[
s \in \{0, 2k, -2k\}.
\]
\end{lemma}

\begin{proof}
Comparing the first and third rows of $C$ gives $c_j = c_{j-2}$ for all $j$ (indices modulo $2k$). Hence, all odd-indexed $c_j$ are equal to some $c \in \{-1,1\}$, and all even-indexed $c_j$ are equal to some $d \in \{-1,1\}$. Therefore, $s = k(c+d)$, which yields $s \in \{0, 2k, -2k\}$. This proves the claim.
\end{proof}

The following two Lemmas are well known; see, e.g., \cite[p.~1193]{HWallis}, \cite[p.~234]{Meisner}, \cite[pp.~329--330]{Turyn}.

\begin{lemma}
\label{regular}
Let $H$ be a regular Hadamard matrix of order $n \geq 4$. Then $n = 4h^2$ for some positive integer $h$. Moreover, if $H$ is circulant then $h$ is odd. Furthermore, either $H$ or $-H$ is $2h$-regular (the other being $(-2h)$-regular). In the $2h$-regular case, each row has $2h^2 + h$ positive entries and $2h^2 - h$ negative entries; in the $(-2h)$-regular case, each row has $2h^2 - h$ positive entries and $2h^2 + h$ negative entries.
\end{lemma}

\begin{lemma}
\label{eigens}
Let $H$ be a circulant Hadamard matrix of order $n$, let $w = \exp(2\pi i / n)$, and let $R(x)$ be its representer polynomial. Then:
\begin{itemize}
\item[\rm(a)]
All eigenvalues $R(s)$ of $H$, where $s \in \{1, w, w^2, \ldots, w^{n-1}\}$, satisfy
\[
|R(s)| = \sqrt{n}.
\]
\item[\rm(b)]
The vector $v := [1,\ldots,1] \in \mathbb{R}^n$ is an eigenvector of $H$ with associated eigenvalue $\lambda = \sqrt{n}$.
\end{itemize}
\end{lemma}

The following lemma is important for the proof of the theorems.

\begin{lemma}
\label{graphR}
Let $H := \circm(h_1,\ldots,h_n)$ be a circulant Hadamard matrix, let $E_1 := \circm(h_1,h_3,h_5,\ldots,h_{n-1})$, and $E_2 := \circm(h_2,h_4,h_6,\ldots,h_{n})$. Then
\begin{equation}
\label{mainx}
\sum_{\substack{j \neq 1 \\ 1 \leq j \leq n/2}} \langle R_{1}, R_{j} \rangle
+ \sum_{\substack{j \neq 1 \\ 1 \leq j \leq n/2}} \langle S_{1}, S_{j} \rangle
= \sum_{\substack{j \neq 1 \\ 1 \leq j \leq n}} \langle T_{1}, T_{j} \rangle = 0,
\end{equation}
where $R_{j}$, $S_{j}$, and $T_{j}$ denote the $j$-th row of $E_1$, $E_2$, and $H$, respectively.
\end{lemma}

\begin{proof}
The result follows from the equality
\begin{equation}
\label{oddeven}
\langle R_{1}, R_{j} \rangle + \langle S_{1}, S_{j} \rangle
= \langle T_{1}, T_{2j-1} \rangle,
\end{equation}
since two distinct rows of $H$ are orthogonal.
\end{proof}

\begin{lemma}
\label{misscase}
Let $H := \circm(h_1,\ldots,h_n)$ be a circulant Hadamard matrix, let $E_1 := \circm(h_1,h_3,h_5,\ldots,h_{n-1})$, and $E_2 := \circm(h_2,h_4,h_6,\ldots,h_{n})$. Let 
\[
\lambda_{1} := h_{1} + h_{3} + \cdots + h_{n-1}, \quad
\lambda_{2} := h_{2} + h_{4} + \cdots + h_{n},
\]
be the real eigenvalues of $E_1$ and $E_2$ associated with the eigenvector $v_{0} := [1,\ldots,1] \in \mathbb{R}^{n/2}$. Then $\lambda_1 \lambda_2 = 0$, so that we may assume, without loss of generality (possibly replacing $H$ by $\pi H$), that
\begin{equation}
\label{oddeven2}
h_1 + h_3 + \cdots + h_{n-1} = 0, 
\quad
|h_2 + h_4 + \cdots + h_{n}| = \sqrt{n}.
\end{equation}
\end{lemma}

\begin{proof}
With the same notation as in Lemma~\ref{graphR}, put 
\[
a := \sum_{\substack{j \neq 1 \\ 1 \leq j \leq n/2}} \langle R_{1}, R_{j} \rangle, 
\quad
b := \sum_{\substack{j \neq 1 \\ 1 \leq j \leq n/2}} \langle S_{1}, S_{j} \rangle.
\]
Since $h_j^2 = 1$ for all $j$, we have
\[
\langle R_{1}, R_{1} \rangle = \frac{n}{2} = \langle S_{1}, S_{1} \rangle.
\]
Thus,
\[
a + \frac{n}{2} = \langle R_{1}, \sum_{1 \leq j \leq n/2} R_j \rangle
= \langle R_{1}, \lambda_1 v_{0} \rangle = \lambda_1^2,
\]
and
\[
b + \frac{n}{2} = \langle S_{1}, \sum_{1 \leq j \leq n/2} S_j \rangle
= \langle S_{1}, \lambda_2 v_{0} \rangle = \lambda_2^2.
\]
By Lemma~\ref{graphR} we have $a + b = 0$. Adding the two equations above gives
\[
\lambda_1^2 + \lambda_2^2 = n.
\]
Moreover, Lemma~\ref{eigens} implies
\[
\lambda_{1} + \lambda_{2} = h_1 + h_2 + \cdots + h_n \in \{\sqrt{n}, -\sqrt{n}\}.
\]
These two relations together yield $\lambda_1 \lambda_2 = 0$, as claimed.
\end{proof}

In order to prove our second theorem we will recall a special case of the Plotkin bound (see \cite{Plotkin}, \cite{JornQ}) as well as a simple corollary of it involving Hadamard matrices.

\begin{lemma}[Plotkin bound, special case]
\label{plotkin}
Let $V = \mathbb{F}_2^k$ and let $C$ be a finite subset of a subspace $W$ of dimension $m$. Let $d$ be the minimum distance in $C$, namely
\[
d = \min_{\substack{x,y \in C \\ x \neq y}} H(x,y),
\]
where $H(x,y)$ denotes the Hamming distance between $x$ and $y$, i.e., the number of coordinates equal to $1$ in $x+y$. Let $A_2(m,d)$ be the maximum size of a subset $S$ of some $m$-dimensional subspace of $V$ having minimum distance $d$. If $d$ is even and $2d > m$, then
\[
A_2(m,d) \leq 2 \left\lfloor \frac{d}{2d - m} \right\rfloor.
\]
\end{lemma}

\begin{lemma}
\label{equal1HAD}
Let $H = (H_{i,j})$ be a Hadamard matrix of order $n>2$. 
If a submatrix $L$ of $H$ of size $a \times b$
contains only $1$'s (or only $-1$'s), then $ab \leq n$.
\end{lemma}

\begin{proof}
Without loss of generality, assume that all entries of $L$ are equal to $1$. 
Let $1 \leq i_1,\ldots,i_{a} \leq n$ and $1 \leq j_1, \ldots,j_{b} \leq n$ be indices such that 
the entries $(L_{r,s})$ of $L$ with $1 \leq r \leq a$ and $1 \leq s \leq b$ satisfy 
\[
L_{r,s} = H_{i_r,j_s}.
\]
Let $R_{i_1},\ldots,R_{i_{a}}$ be the corresponding rows of $H$ and 
$T_{j_1},\ldots,T_{j_{b}}$ the corresponding columns of $H$. 
Put $C = \{R_{i_1},\ldots,R_{i_{a}}\}$.

For any pair of rows $R_{i_h}, R_{i_k}$ in $C$ (say $R_{i_1},R_{i_2}$), define:  
(a) $x_{1,2}$ as the number of columns $T_k$ with $k \notin \{j_1,\ldots,j_{b}\}$ in which $H_{i_1,k} = H_{i_2,k}$;  
(b) $y_{1,2}$ as the number of such columns in which $H_{i_1,k} \neq H_{i_2,k}$.  
Thus,
\begin{equation}
\label{suma}
x_{1,2}+y_{1,2} = n - b,
\end{equation}
while the orthogonality of $R_{i_1}$ and $R_{i_2}$ implies
\begin{equation}
\label{diferencia}
b + x_{1,2} - y_{1,2} = \langle R_{i_1},R_{i_2}\rangle = 0.
\end{equation}
From \eqref{suma} and \eqref{diferencia} we obtain
\begin{equation}
\label{y12}
y_{1,2} = \frac{n}{2}.
\end{equation}
In other words, \eqref{y12} says that
\begin{equation}
\label{H12}
H(R_{i_1},R_{i_2}) = \frac{n}{2},
\end{equation}
since all elements of $L$ are equal to $1$.

We have thus proved that the minimum Hamming distance $d$ in $C$ satisfies
\begin{equation}
\label{denC}
d = \frac{n}{2}.
\end{equation}
Let $m = n-b$. With the notation of Lemma \ref{plotkin}, we have
\begin{equation}
\label{Adosa}
A_2(m,d) = a.
\end{equation}

Since $H$ is Hadamard and $n>2$, the order $n$ is a multiple of $4$. Hence \eqref{denC} implies that $d$ is even. Moreover, $2d = n > m = n - b$, so that \eqref{Adosa} and Lemma \ref{plotkin} imply
\begin{equation}
\label{aplot}
a \leq 2 \left\lfloor \frac{n/2}{n - (n - b)} \right\rfloor 
= 2 \left\lfloor \frac{n/2}{b} \right\rfloor.
\end{equation}

Multiplying \eqref{aplot} by $b$ we obtain
\begin{equation}
\label{aplot2}
a b \leq 2b \left\lfloor \frac{n/2}{b} \right\rfloor \leq n.
\end{equation}
This proves the result.
\end{proof}

\section{Proof of Theorem \ref{HadBM}}
\label{HaddoneBM}

\begin{proof}
By Lemma \ref{crankreal}, it suffices to prove parts (c) and (d).

We first prove (c):  
By Lemma \ref{circrank1}, we can assume that the first row of $E_1$ has exactly $h^2$ entries equal to $1$ and $h^2$ entries equal to $-1$. This implies that $E_2$ has $h^2+h$ entries in its first row equal to $1$ and $h^2-h$ entries equal to $-1$. On the other hand, either all entries in the first row of $E_2$ are equal, in which case we must have $h^2+h=2h^2$, i.e., $h^2=h$, hence $h=1$, or half of them are equal to $1$, in which case we must have either $h^2+h=h^2$ or $h^2-h=h^2$, both of which are impossible. Therefore, $h=1$, and hence $n=4$.

To prove part (d), put $R_1 := [k_1,\ldots,k_{n/2}]$. Observe that $\langle R_{1},R_{1}\rangle = n/2$. Lemma \ref{cauchyschw}, applied to each of the terms $k_j k_{j-1}$ that sum up to $\langle R_{1},R_{2}\rangle$ (for indices $j=1,\ldots,n/2+1$, with $k_{n/2+1}=k_1$), shows that all $k_j$ with odd $j$ are equal, say to $c_1 \in \{-1,1\}$ (there are exactly $h^2$ such values of $j$), and that all $k_j$ with even $j$ are equal, say to $c_2 \in \{-1,1\}$ (there are also exactly $h^2$ such values of $j$, since $n/2 = 2h^2$). By Lemma \ref{circrank1}, we can assume without loss of generality that $c_1+c_2=0$. Thus, the first row of $E_1$ has exactly $h^2$ entries equal to $1$ and $h^2$ entries equal to $-1$.

Using the same argument with $S_1$, it remains to consider the following case: Lemma \ref{circrank1} implies that the first row of $E_2$ has $h^2$ entries equal to $1$ and $h^2$ entries equal to $1$, i.e., $c_1=c_2$. In other words, all entries of $E_2$ are equal to $1$. This implies that $E_1$ must contain $h$ entries equal to $1$ and $2h^2-h$ entries equal to $-1$. This is impossible when $h>1$ by Lemma  \ref{circrank1} since $E_1$ has order $2h^2$ and rank $1$. Therefore, $h=1$, and hence $n=4$.
\end{proof}

\section{Proof of Theorem \ref{PlotkinR}}
\label{Hdone2}

\begin{proof}
Assume, by contradiction, that $H$ is a circulant Hadamard matrix with $n > 4$.

We first prove part (a).  
By Lemma \ref{regular}, we have $n = 4h^2$ with $h$ odd. Assume, without loss of generality, that $E_1$ has rank $1$. By Lemma \ref{circrank1}, we can assume, without loss of generality, either:

(i) $h_1 = h_3 = \ldots = h_{n-1}$, or  
(ii) $h_1 = 1$, $h_3 = -1$, $\ldots$, $h_{n-3} = 1$, $h_{n-1} = -1$.

When (i) holds, we consider the submatrix $L$ of $H$ equal to $E_1$. More precisely, the entries of $L$ are those of $E_1$, at the intersection of rows $1,3,5,\ldots,4h^2-1$ and columns $1,3,5,\ldots,4h^2-1$ of $H$. Since $L$ has size $2h^2 \times 2h^2$ and all its entries are equal to $1$, Lemma \ref{equal1HAD} implies that
\begin{equation}
\label{abbound}
2h^2 \cdot 2h^2 \leq 4h^2.
\end{equation}
In other words, \eqref{abbound} says that $h^2 \leq 1$. But $h$ is odd, so $h=1$, and $n=4$, contradicting our assumption $n>4$. This proves part (i) of (a).

Similarly, when (ii) holds, we consider the submatrix $L$ of $H$ equal to the submatrix $F_1 = \circm(h_1,h_5,\ldots,h_{4h^2-5},h_{4h^2-1})$ of $E_1$. The entries of $L$ are those of $F_1$ at the intersection of rows $1,5,\ldots,4h^2-5,4h^2-1$ and columns $1,5,\ldots,4h^2-5,4h^2-1$ of $H$. Since $L$ has size $h^2 \times h^2$ and all its entries are equal to $1$, Lemma \ref{equal1HAD} implies that
\begin{equation}
\label{abbnew}
h^2 \cdot h^2 \leq 4h^2,
\end{equation}
so that $h^2$ is an odd number with $h^2 \leq 4$, i.e., $h=1$. As before, we get the contradiction $n=4$. This proves part (a) of the theorem.

To prove part (b), assume that both $E_1$ and $E_2$ have rank $2$. By Lemma \ref{regular}, $n = 4h^2$ with $h$ odd. Then, by Lemma \ref{misscase} (and changing $H$ to $-H$ if necessary), we may assume that
\begin{equation}
\label{raizn}
h_2+h_4+\cdots+h_{4h^2} = 2h.
\end{equation}

Let $v,w$ be the first two rows of $E_2 = \circm(h_2,h_4,\ldots,h_{4h^2})$. We claim that $v$ and $w$ are $\Q$-linearly independent. Otherwise, putting $C=E_2$ in Lemma \ref{circrank1} yields
\begin{equation}
\label{iguales}
h_2 = h_4 = \cdots = h_{4h^2} = 1.
\end{equation}
From \eqref{iguales} and \eqref{raizn} we get
\begin{equation}
\label{hunoagain}
2h^2 = 2h,
\end{equation}
i.e., the contradiction $h=1$. This proves the claim.

Let $u$ be the third row of $E_2$. Since $v$ and $w$ form a basis of the $\Q$-vector space generated by the rows of $E_2$, there exist rational numbers $a,b$ such that
\begin{equation}
\label{losab}
u = av + bw.
\end{equation}
Let $S = h_2+h_4+\cdots+h_{4h^2}$. Since, by \eqref{raizn}, $S \neq 0$, Lemma \ref{circrank2} with $D=E_2$ and $s=S$ implies
\begin{equation}
\label{ab1uno1}
a+b = 1.
\end{equation}

Let $h_{2j}$ be any entry of $u$, with $j=2h^2-1, 2h^2, 1, 2, \ldots, 2h^2-2$. By \eqref{losab} we have
\begin{equation}
\label{ab2h}
h_{2j} = a h_{2j+4} + b h_{2j+2}.
\end{equation}
Assume $ab \neq 0$. By Lemma \ref{klmrank2}, \eqref{ab2h} implies
\begin{equation}
\label{tresig}
h_{2j} = h_{2j+4} = h_{2j+2}
\end{equation}
for all such $j$. In other words, all entries of $v$ are equal. By \eqref{raizn} they must be $1$, and so \eqref{raizn} reads exactly as \eqref{hunoagain}, giving again the contradiction $h=1$. Therefore $ab=0$. By \eqref{ab1uno1}, the only cases in \eqref{losab} are $w=u$ when $(a,b)=(0,1)$ and $v=u$ when $(a,b)=(1,0)$.

Since $w$ and $u$ are consecutive rows of $E_2$, the same argument used above to show $v$ and $w$ are $\Q$-linearly independent applies to show $w$ and $u$ are also $\Q$-linearly independent. Thus, $w \neq u$, so the case $w=u$ is impossible. Assume then $v=u$. By Lemma \ref{uvuu} with $C=E_2$ and $s=S$, we obtain
\begin{equation}
\label{h2h2}
S \in \{0, 2h^2, -2h^2\}.
\end{equation}
But by \eqref{raizn}, $S = 2h$. From \eqref{h2h2} it follows that \eqref{hunoagain} still holds, yielding again the contradiction $h=1$.

This finishes the proof of the theorem.
\end{proof}











\end{document}